\documentclass[a4paper,9pt,twoside]{article}

\usepackage{amsmath, amssymb, amsthm}
\usepackage{color}

\newtheorem{thm}{Theorem}

\newtheorem{lemma}{Lemma}
\newtheorem{cor}{Corollary}
\newtheorem{remark}{Remark}

\numberwithin{equation}{section}
\allowdisplaybreaks

\def\dis{\displaystyle}
\def\R{\mathbb{R}}

\def\N{\mathbb{N}}

\def\e{{\epsilon}}
\def\D{\Delta}

\def\l{\left}
\def\r{\right}
\def\a{\alpha}
\def\b{\beta}
\def\G{\Gamma}
\def\g{\gamma}
\def\th{\theta}
\def\s{\sigma}
\def\z{\zeta}
\def\k{\kappa}
\def\la{\lambda}

\def\n{\nabla}
\def\p{\partial}

\def\P{{\mathbb{P}}}
\def\toP{\stackrel{\P}{\longrightarrow}}
\def\toD{\stackrel{\mathcal{D}}{\longrightarrow}}
\def\E{\mathbb{E}}

\def\mb#1{\mbox{\boldmath $#1$}}

\def\I{\mathbf{1}}

\def\L{\mathcal{L}}

\def\wh#1{\widehat{#1}} 
\def\ol#1{\overline{#1}} 
\def\wt#1{\widetilde{#1}} 
\def\df{{\rm d}}

\title{Confidence intervals of ruin probability under L\'evy surplus}
\author{Yasutaka Shimizu \\ 
{\it Department of Applied Mathematics} \\{\it Waseda University} \\ E-mail: {\tt shimizu@waseda.jp} }
\date{\today}

\begin{document}

\maketitle 

\begin{abstract} 

The aim of this paper is to construct the confidence interval of the ultimate ruin probability under the insurance surplus driven by a L\'evy process. Assuming a parametric family for the L\'evy measures, we estimate the parameter from the surplus data, and estimate the ruin probability via the {\it delta method}. However the asymptotic variance includes the derivative of the ruin probability with respect to the parameter, which is not generally given explicitly, and the confidence interval is not straightforward even if the ruin probability is well estimated. This paper gives the Cram\'er-type approximation for the derivative, and gives an asymptotic confidence interval of ruin probability.

\begin{flushleft}
{\it Key words:} L\'evy insurance risks; ruin probability; Cram\'er approximation; delta method; asymptotic confidence interval. \vspace{1mm}\\
{\it MSC2010:} 62G20 (91B30, 62M05).
\end{flushleft}
\end{abstract}

\section{Introduction}
\subsection{Ruin probability under L\'evy surplus}

Consider the following insurance surplus:  
\begin{align}
R_t = u + ct + \s W_t - S_t,\quad t\ge 0, \label{model}
\end{align}
where $u,c\ge 0$ are constants, $W$ is a standard Brownian motion, and $S$ is a L\'evy subrdinator with L\'evy measure $\nu_\a$ on $(0,\infty)$ satisfying that 
$\int_0^1 z\nu_\a(\df z) < \infty$, and $\a$ is a parameter. The Laplace transform of $S$ is given by 
\[
\E\l[e^{-u S_t}\r] = \exp\l(t \int_0^\infty \l(1 - e^{-uz}\r)\,\nu_\a(\df z)\r),\quad u\ge 0, 
\]
and the mean of $S_1$ exists: 
\[
m_\a := \E[S_1] = \int_0^\infty z\,\nu_\a(\df z) < \infty. 
\]
As is well known, the following condition is needed to avoid the almost sure ruin. 
\begin{flushleft}
{\bf [NPC]} ({\it the net profit condition})
\[
c > m_\a. 
\]
\end{flushleft}

We use the following notation throughout the paper: 
\[
D=\s^2/2,\quad \Pi_\a(z) =\nu_\a(z,\infty), 
\]
and 
\[
\nu_{I,\a}(x) = \int_0^x \Pi_\a(z)\,\df z,\quad \ol{\nu}_{I,\a}(x) = \int_x^\infty \Pi_\a(z)\,\df z. 
\]
Moreover let $k_D$ be the probability density function of the exponential distribution with mean $D/c$: 
\[
k_D(x) = \frac{c}{D} e^{-\frac{c}{D}x},\quad x\ge 0. 
\]
Letting $\th=(\a,D)$, we consider the ultimate ruin probability give by 
\[
\psi_\th(u) = \P\l(\inf_{t\ge 0}R_t < 0\,\Big|\,R_0=u\r), \quad u\ge 0. 
\]

Hereafter the notation $\star$ stands for the convolution (of Lebesgue integral)
\[
f\star g(u) = \int_0^u f(u-x)g(x)\,\df x
\]
for functions $f,g\,:\,(0,\infty) \to \R$. 

According to Huzack {\it et al.} \cite{hetal04}, Theorem 3.1, the following {\it defective renewal equation} for the ruin probability $\psi_\th$ is straightforward; see also Biffis and Morales \cite{bm10}, for more general case dealt with the {\it Gerber-Shiu function}. 

\begin{thm}\label{thm:dre-psi}
Under the condition [NPC], it holds that 
\begin{align}
\psi_\th(u) = \psi_\th\star g_\th(u) + h_\th(u),\quad u\ge 0, \label{dre-psi}
\end{align}
where 
\[
g_\th(u) = \frac{1}{c} k_D\star \Pi_\a(u),\quad h_\th(u) = \frac{1}{c}k_D\star \ol{\nu}_{I,\a}(u) + \int_u^\infty k_D(z)\,\df z
\]
\end{thm}

Consider the {\it modified Lundberg equation}: 
\begin{align}
\k_\th(r):= -cr + D r^2 + \int_0^\infty (e^{rz} - 1)\,\nu_\a(\df z) = 0. \label{l-eq}
\end{align}
(Note that so-called the ``Lundberg equation" is ``$\log \E[e^{r(R_1-u)}]=\k_\th(-r)=0$") 
Provided [NPC] and that, for a sufficiently large $r>0$, 
\[
\int_1^\infty e^{rz}\,\nu_\a(\df z) < \infty, 
\]
the positive solution to the equation exists, say $r = \g_\th>0$, and it is called {\it adjustment coefficient}. 
\begin{remark}\label{rem:gcd}
If 
\begin{align}
\int_1^\infty e^{\frac{c}{D}z}\,\nu_\a(\df z) < \infty, \label{gcd-cond}
\end{align}
then the adjustment coefficient $\g_\th$ exists since it holds that 
\[
\k_\th(0)=\k_\th(\g_\th)=0;\quad \k'_\th(0)<0;\quad \k_\th(c/D) = \int_0^\infty \l(e^{\frac{c}{D}z} - 1\r)\nu_\a(\df z) > 0. 
\]
These facts implies that 
\begin{align}
0< \g_\th < \frac{c}{D}, \label{gcd}
\end{align}
which also holds true if $\g_\th$ exists unless \eqref{gcd-cond} holds. 
\end{remark}

Multiplying $e^{\g_\th u}$ to the both sides of \eqref{dre-psi}, 
\begin{align}
\wt{\psi}_\th(u) = \wt{\psi}_\th\star \wt{g}_\th(u) + \wt{h}_\th(u), \label{z-tilde}
\end{align}
where $\wt{f}(u)= e^{\g_\th u}f(u)$ for the function $f$. 
Since it is easily checked (see the proof of Theorem \ref{thm:cl-approx}) that 
\begin{align}
\int_0^\infty \wt{g}_\th(z)\,\df z = 1,  \label{g-tilde}
\end{align}
the equation \eqref{z-tilde} is the (proper) renewal equation. 
Hence the usual argument using ``Key Renewal Theorem" yields that the following {\it Cram\'er-type approximation} of the ruin probability. 
\begin{thm}\label{thm:cl-approx}
Suppose the condition [NPC] and the adjustment coefficient $\g_\th>0$ exists. 
Moreover, suppose the moment condition such as 
\[
\int_1^\infty ze^{\g_\th z}\,\nu_\a(\df z) < \infty. 
\]
Then it holds that 
\begin{align}
\psi_\th(u) \sim C_\th\,e^{-\g_\th u},\quad u\to \infty, \label{cl-approx}
\end{align}
where 
\[
C_\th= \frac{c-m_\a}{\dis \int_0^\infty ze^{\g_\th z}\,\nu_\a(\df z) - c + 2D\g_\th}.
\]
\end{thm}

\begin{proof}
This result is not new, but we will describe the proof since no explicit proof can be found anywhere. 
Therefore we will give the detailed proof here. 

We start off by the renewal-type equation \eqref{z-tilde}: 
\begin{align*}
\wt{\psi}_\th(u) = \wt{\psi}_\th\star \wt{g}_\th(u) + \wt{h}_\th(u), 
\end{align*}
where 
\begin{align*}
\wt{g}_\th(u) &= \frac{1}{c}e^{\g_\th u} k_D\star \Pi_\a(u), \quad \wt{h}_\th(u) = \frac{1}{c}e^{\g_\th u} k_D\star \ol{\nu}_{I,\a}(u) + e^{(\g_\th - c/D) u}. 
\end{align*}
At the beginning, we will confirm that the function $\wt{g}_\th$ is the probability density. 
By Fubini's theorem, 
\begin{align*}
\int_0^\infty \wt{g}_\th(z)\,\df z &= \frac{1}{c}\int_0^\infty\l(\int_0^z e^{\g_\th (z-x)} k_D(z-x) \cdot e^{\g_\th x}\Pi_\a(x)\,\df x\r)\,\df z \\
&= \frac{1}{D}\int_0^\infty e^{\g_\th x}\Pi_\a(x) \l(\int_x^\infty e^{(\g_\th-c/D) (z-x)} \,\df z\r)\,\df x
\end{align*}
Here we note from Remark \ref{rem:gcd} that \eqref{gcd} at least holds true if $\g_\th$ exists, that is, 
\[
e^{(\g_\th - c/D)x}\to 0,\quad x\to \infty. 
\]
Therefore, using Fubini's theorem again, we have that 
\begin{align*}
\int_0^\infty \wt{g}_\th(z)\,\df z 
&=  \frac{1}{D\g_\th - c}\int_0^\infty e^{\g_\th x} \,\df x\int_x^\infty \nu_\a(\df z)  \\
&= \frac{1}{D\g_\th - c}\int_0^\infty \l(\int_0^z e^{\g_\th x} \,\df x\r)\,\nu_\a(\df z) \\
&= \frac{1}{D\g^2_\th - c\g_\th} \int_0^\infty \l(e^{\g_\th z} - 1\r)\,\nu_\a(\df z) = 1. 
\end{align*}
We used the Lundberg's identity $\k_\th(\g_\th)=0$ for the last equality. Hence the equation \eqref{z-tilde} is the ``proper" renewal equation. 
Then we can apply the ``Key Renewal Theorem" (e.g., Rolski {\it et al.} \cite{retal99}, Theorem 6.1.11) to $\wt{\psi}_\th$ to obtain that 
\[
\lim_{u\to \infty}\wt{\psi}_\th(u) = \frac{\dis \int_0^\infty \wt{h}_\th(x)\,\df x}{\dis  \int_0^\infty x\wt{g}_\th(x)\,\df x}. 
\]
We use Fubini's theorem repeatedly to obtain that 
\begin{align*}
\int_0^\infty x\wt{g}_\th(x)\,\df x &= \frac{1}{\g_\th(c-D\g_\th)} \int_0^\infty z e^{\g_\th z}\,\nu_\a(\df z) + \frac{c-2D\g_\th}{\g_\th^2(c-D\g_\th)^2} \int_0^\infty \l(e^{\g_\th z} - 1\r)\,\nu_\a(\df z), \\
\int_0^\infty \wt{h}_\th(x)\,\df x&= \frac{1}{\g_\th(c-D\g_\th)}\l[\frac{1}{\g_\th}\int_0^\infty \l(e^{\g_\th z} - 1\r)\,\nu_\a(\df z) - m_\a\r] + \frac{D}{c-D\g_\th}
\end{align*}
Using the Lundberg identity: 
\[
\k_\th(\g_\th) = 0\quad \Leftrightarrow\quad c\g_\th - D\g_\th^2 = \int_0^\infty (e^{\g_\th z} - 1)\,\nu_\a(\df z),
\]
we obtain that 
\begin{align*}
\lim_{u\to \infty}\wt{\psi}_\th(u) &= \frac{\frac{1}{\g_\th(c-D\g_\th)} \l(c-D\g_\th - m_\a\r) + \frac{D}{c-D\g_\th}}{\frac{1}{\g_\th(c-D\g_\th)}\int_0^\infty z e^{\g_\th z}\,\nu_\a(\df z) - +\frac{c-2D\g_\th}{\g_\th}} \\
&=  \frac{c-m_\a}{\int_0^\infty ze^{\g_\th z}\,\nu_\a(\df z) - c + 2D\g_\th}. 
\end{align*}
This completes the proof. 
\end{proof}

\subsection{Parametric inference for the estimator of ruin probability}\label{sec:par-inf}

To consider the parametric inference for the ruin probability, we prepare a parametric family for the L\'evy measures 
\[
{\cal P}_\Xi = \l\{\nu_\a\,|\,\a\in \Xi\r\},
\]
where $\Xi$ is a subset of $\R^p$, and consider the parameter space 
\[
\th = (\a,D) \in \Xi\times \Lambda =:\Theta. 
\]
We assume that $\Theta$ is a bounded, compact and convex subset of $\R^{p+1}$. 
We further assume that the true value of the parameter $\th$, say $\th_0$, belongs to the interior of $\Theta$: 
\[
\th_0 = (\a_0,D_0)\in int(\Theta),\quad D_0 := \frac{\s_0^2}{2}. 
\]

Now, suppose that the surplus process $R$ is observed in $[0,T]$-time interval. 
Although there are many possibilities for the sampling scheme of the L\'evy surplus, which will be described later, 
we now suppose that a consistent estimator for $\th_0$ is obtained based on some data set in $[0,T]$, say 
\[
\wh{\th}_T \toP \th_0,\quad T\to \infty. 
\]
Furthermore, suppose that a limit in law of $\wh{\th}_T$, say $Z$, is found as follows: 
\[
\sqrt{T}(\wh{\th}_T - \th_0) \toD Z, \quad T\to \infty. 
\]
Then it holds by the {\it delta method} that 
\[
\sqrt{T}(\psi_{\wh{\th}_T} - \psi_{\th_0}) \toD \dot{\psi}_\th^\top Z,\quad T\to \infty, 
\]
where $\dot{\psi}_\th = \l(\frac{\p\psi_{\th}}{\p \a_1},\dots, \frac{\p\psi_{\th}}{\p {\a_p}},\frac{\p \psi_\th}{\p D}\r)^\top$. 
To construct the confidence interval of $\psi_{\th_0}$, we need to estimate  the derivative $\dot{\psi}_{\th_0}$. 
However it does not have the closed expression as similarly to $\psi_{\th}$ as seen later. 

In this paper, we will investigate the Cram\'er-type asymptotic formula for $\dot{\psi}_{\th}(u)$ as $u\to \infty$, 
and obtain a non trivial limit in law of $\sqrt{T}(\psi_{\wh{\th}_T}(u) - \psi_{\th_0}(u))$ as $T\to \infty$ and $u\to \infty$ at the same time.

\section{On the derivatives $\dot{\psi}_\th$ and $\ddot{\psi}_\th$}
\subsection{Notation and assumption}

We make the notation used in this paper. 

\begin{itemize}
\item For functions $f,\,g$, we write $f(u)\sim g(u),\quad u\to \infty$ if $\dis \lim_{u\to \infty}f(u)/g(u) = 1$. 

\item For positive functions $f,\,g$, we write $f(u)\lesssim g(u)$ if there exists a constant $C>0$ such that $f(u)\le C\cdot g(u)$ for any $u$. 
In particular, $f(u)\lesssim 1$ implies that $f$ is uniformly bounded with respect to $u$.

\item For $l=0,1,2,\dots,$ we denote by $\n_\th^l f$ the $l(p+1)$-dimensional array of the partial derivatives of $f$ with respect to the components of $\th$. 
In particular, 
\begin{align*}
\dot{f}_\th &= \n_\th f_\th = \l(\frac{\p f_\th}{\p \a_1},\dots,\frac{\p f_\th}{\p \a_p}, \frac{\p f_\th}{\p D},\r)^\top, \\
\ddot{f}_\th&= \n_\th^2f_\th=\begin{pmatrix} \l(\frac{\p^2 f_\th}{\p \a_i\p{\a_j}}\r)_{i,j} & \l(\frac{\p^2 f_\th}{\p \a_i\p{D}}\r)_{i} \\\l(\frac{\p^2 f_\th}{\p D\p{\a_j}}\r)_{j} & \frac{\p^2 f_\th}{\p D^2}\end{pmatrix}. 
\end{align*}

\item $\mb{O}_{ij}$ is the $i\times j$ zero matrix. 

\item $\|x \|$ is the Eucridian norm of $x$ in whichever space it lies. 

\item For functions $f=(f_1,\dots,f_k)^\top:\R_+\to \R^k$ and $g=(g_1,\dots,g_l)^\top:\R_+\to \R^l$, 
\[
f\star g(u) = \l(f_i\star g_j(u)\r)_{1\le i\le k,1\le j\le l} \in \R^k\otimes \R^l
\]

\item Laplace transform: for a function $g:\R_+\to \R^k$ with $\int_0^\infty |g(x)|\,\df x < \infty$, 
\[
\L g(r) = \int_0^\infty e^{-r x}g(x)\,\df x,\quad r\ge 0.  
\]
Note that the integral is elementwise if $k\ge 2$. 

\end{itemize}

To investigate the derivatives of $\psi_\th$ with respect to $\th$, we make the following assumptions. 

\begin{flushleft}
{\bf [MO]}\ \ For each $\th\in \Theta$, there exist the adjustment coefficient $\g_\th>0$ and some $\e>0$ such that $\dis \int_1^\infty e^{(\g_\th + \e)z}\nu_\a(\df z) < \infty$.  
\end{flushleft}

\begin{flushleft}
{\bf [TD($k$)]}\ \ For each $x>0$, $\Pi_\cdot(x) \in C^k(\Xi)$. Moreover, there exists a constant $b_\th > \g_\th$ for each $\th\in \Theta$ such that, for $l=0,1,\dots,k$, 
\[
\l\|\nabla_\a^l\Pi_\a(x)\r\| \lesssim e^{-b_\th x},\quad b:=\inf_{\th\in \Theta}b_\th > 0. 
\]
\end{flushleft} 

\begin{flushleft}
{\bf [DI($k$)]}\ \ For each $u>0$, 
\[
\n_\th^k \int_0^u \psi_\th(u-z)g_\th(z)\,\df z = \int_0^u \n_\th^k \l[\psi_\th(u-z)g_\th(z)\r]\,\df z
\]
\end{flushleft}

\subsection{Asymptotic formulae for the derivatives}
Differentiate the renewal equation \eqref{dre-psi} formally in the both sides with respect to the parameter $\th$, we have that, for each $u\ge 0$, 
\begin{align}
\dot{\psi}_\th(u) = \dot{\psi}_\th \star g_\th(u) + \l[\dot{h}_\th(u) + \psi_\th\star \dot{g}_\th(u)\r]  \label{dre-1}
\end{align}
and 
\begin{align}
\ddot{\psi}_\th(u) = \ddot{\psi}_\th \star g_\th(u) + \l[\ddot{h}_\th(u) + 2\dot{\psi}_\th\star \dot{g}_\th(u) + \psi_\th\star \ddot{g}_\th(u)\r].   \label{dre-2}
\end{align}

Later we need the asymptotic formula for $\dot{\psi}_\th$ and $\ddot{\psi}_\th$ as $u\to \infty$. For that purpose, 
we introduce a version of {\it Key Renewal Theorem}. 

\begin{lemma}[Oshime and Shimizu \cite{os17}, Theorem 4]\label{lem:key}
Suppose that a function $Z$ meets the renewal equation
\[
Z(u) = Z*F(u) + G(u),\quad u\ge 0, 
\]
where $F$ is a probability distribution with mean $\mu_F$, 
and $G$ is a function with bounded variation and satisfies that $\sup_{u\in [0,1]}|G(u)|\lesssim 1$. 
If there exists 
\[
A := \lim_{u\to \infty}\frac{G(u)}{u^{k-1}} < \infty
\]
for some $k\in \N$, it holds that 
\[
Z(u) \sim \frac{A}{k \mu_F}u^k,\quad u\to \infty.
\]
\end{lemma}

According to \eqref{g-tilde}, we can easily transform \eqref{dre-1} and \eqref{dre-2} into the proper renewal equation by multiplying the term $e^{\g_\th u}$. 
The straightforward application of Lemma \ref{lem:key} with $k=1$ or $2$ immediately yields the following theorem. 

\begin{thm}\label{thm:psi-deriv}
Under the conditions [NPC], [MO], [TD(1)] and [DI(1)], it holds that 
\begin{align}
\dot{\psi}_\th(u) \sim C_\th \l[\mu_\th^{-1}\L \dot{g}_\th(-\g_\th)\r] u e^{-\g_\th u},\quad u\to \infty. \label{psi-1}
\end{align}
where $\mu_\th = \int_0^\infty x \wt{g}_\th(x)\,\df x = -(\L g_\th)'(-\g_\th)$ and $C_\th$ is the constant given in Theorem \ref{thm:cl-approx}. 
In particular, $\dot{g}_\th = c^{-1}\big(k_D\star \n_\a\Pi_\a, \n_D k_D\star \Pi_\a\big)^\top$. 

In addition that, suppose further the conditions [TD(2)] and [DI(2)], it is true that 
\begin{align}
\ddot{\psi}_\th(u) \sim C_\th \l[\mu_\th^{-1}\L \dot{g}_\th(-\g_\th)\r]^{\otimes 2} u^2 e^{-\g_\th u},\quad u\to \infty. \label{psi-2}
\end{align}
\end{thm}

\begin{proof}
In the beginning, we note that the equation \eqref{dre-1} is valid since we can differentiate the both sides under the integral sign by [DI(1)]. 
The differentiability of $h_\th$ is also confirmed from [TD(1)]. 
Then, multiplying $e^{\g_\th u}$ in the both sides of \eqref{dre-1}, we have that 
\begin{align}
\wt{\dot{\psi}}_\th(u) = \wt{\dot{\psi}}_\th \star \wt{g}_\th(u) + \l[\wt{\dot{h}}_\th(u) + \wt{\psi}_\th\star \wt{\dot{g}}_\th(u)\r], \label{RE-dot-psi}
\end{align}
which is the proper renewal equation due to \eqref{g-tilde}. Applying Lemma \ref{lem:key} with $k=1$ to this renewal equation, it suffices to show that 
\begin{align}
\lim_{u\to \infty}  \wt{\dot{h}}_\th(u) &= 0; \qquad 
\lim_{u\to \infty}  \wt{\psi}_\th\star \wt{\dot{g}}_\th(u) =  C_\th \L \dot{g}_\th(-\g_\th),  \label{property:RE-dot-psi}
\end{align}
for the proof of \eqref{psi-1}. Note that 
\begin{align*}
\wt{\dot{h}}_\th(u) &= e^{\g_\th u} \n_\th\l[\frac{1}{c}\int_0^u k_D(u-x)\ol{\nu}_{I,\a}(x)\,\df x + e^{-(c/D)u} \r] \\
&= e^{\g_\th u} \frac{1}{c}\int_0^u \begin{pmatrix} k_D(u-x) \n_\a \ol{\nu}_{I,\a}(x) \\ \n_D k_D(u-x)\ol{\nu}_{I,\a}(x)\end{pmatrix}\,\df x 
+ \begin{pmatrix}0 \\ \frac{c}{D^2} e^{(\g_\th -c/D)u} \end{pmatrix}  \\
&= \int_0^u \begin{pmatrix} \frac{1}{D}e^{(\g_\th - c/D)(u-x)} e^{\g_\th x}\n_\a \ol{\nu}_{I,\a}(x) \\ \frac{1}{D^2}\wt{\ol{\nu}}_{I,\a}(u-x) (x-1)e^{(\g_\th - c/D)x}\end{pmatrix}\,\df x 
+ \begin{pmatrix}0 \\ \frac{c}{D^2} e^{(\g_\th -c/D)u} \end{pmatrix} \\
&\lesssim \int_0^u \begin{pmatrix} e^{(\g_\th - c/D)(u-x)} e^{(\g_\th-b_\th) x} \\ e^{(\g_\th-b_\th)(x-u)} (x-1)e^{(\g_\th - c/D)x}\end{pmatrix}\,\df x 
+ \begin{pmatrix}0 \\ e^{(\g_\th -c/D)u} \end{pmatrix}. 
\end{align*}
We also note that the exchangeability of differential and the integral sign in the second equality is ensured under [TD(1)] since 
\begin{align*}
 \sup_{\a\in \Xi}\l|k_D(u-x) \n_\a \ol{\nu}_{I,\a}(x)\r| &\lesssim k_D(u-x)e^{-b x}, \\
\sup_{D\in \Lambda}\l|\n_D k_D(u-x) \ol{\nu}_{I,\a}(x)\r|&\lesssim (x+1)e^{-\frac{c}{D^*}(x-u)}e^{-bx}, 
\end{align*}
where $D^*:=\max \Lambda$. Since 
\begin{align}
\g_\th -\frac{c}{D} < 0,\quad b_\th > \g_\th \label{cond:g-b} 
\end{align}
from \eqref{gcd} and the assumption, we see the last term goes to zero by the Lebesgue convergence theorem. 
Moreover, 
\begin{align*}
 \wt{\psi}_\th\star \wt{\dot{g}}_\th(u) &= \int_0^\infty e^{\g_\th (u-x)}\psi_\th(u-x)\I_{\{x\le u\}}\cdot e^{\g_\th x}\dot{g}_\th(x)\,\df x. 
\end{align*}
By Theorem \ref{thm:cl-approx} and the fact that 
\[
\l\|e^{\g_\th x} \dot{g}_\th(x)\r\| \lesssim  e^{(\g_\th - c/D - b_\th)x}, 
\]
we have that 
\[
\l\|e^{\g_\th (u-x)}\psi_\th(u-x)\I_{\{x\le u\}}\cdot e^{\g_\th x}\dot{g}_\th(x)\r\| \lesssim e^{(\g_\th - c/D - b_\th)x}, 
\]
which is integrable by \eqref{cond:g-b}, and is independent of $u$. Then the Lebesgue convergence theorem yields that 
\begin{align*}
\lim_{u\to \infty} \wt{\psi}_\th\star \wt{\dot{g}}_\th(u) 
= \int_0^\infty C_\th \wt{\dot{g}}_\th(x)\,\df x  = C_\th \L \dot{g}_\th(-\g_\th). 
\end{align*}
This ends the proof of \eqref{psi-1}. 

We shall take the similar argument to show \eqref{psi-2}, that is, we start off the proper renewal equation \eqref{psi-2}: 
\begin{align*}
\wt{\ddot{\psi}}_\th(u) = \wt{\ddot{\psi}}_\th \star \wt{g}_\th(u) + \l[\wt{\ddot{h}}_\th(u) + 2\wt{\dot{\psi}}_\th\star \wt{\dot{g}}_\th(u) + \wt{\psi}_\th\star \wt{\ddot{g}}_\th(u)\r], 
\end{align*}
which is valid under [TD(2)] and [DI(2)]. By the same argument as above, it is easy to see the following: 
\begin{align*}
\lim_{u\to \infty}  u^{-1} \l[\wt{\ddot{h}}_\th(u) + \wt{\psi}_\th\star \wt{\ddot{g}}_\th(u)\r]= 0. 
\end{align*}
Moreover, using \eqref{psi-1}, we have that 
\begin{align*}
\lim_{u\to \infty}u^{-1} \wt{\dot{\psi}}_\th\star \wt{\dot{g}}_\th(u) 
&= \int_0^\infty \l[\lim_{u\to \infty}  u^{-1}e^{\g_\th (u-x)} \dot{\psi}_\th(u-x)\r]\, e^{\g_\th x}\dot{g}_\th(x)\,\df x \\
&= \frac{C_\th}{\mu_\th} \l[\L \dot{g}_\th(-\g_\th)\r]^{\otimes 2}. 
\end{align*}
As a result, Lemma \ref{lem:key} with $k=2$ yields that 
\[
\lim_{u\to \infty} u^{-2}\wt{\ddot{\psi}}_\th(u) = C_\th \l[\mu_\th^{-1}\L \dot{g}_\th(-\g_\th)\r]^{\otimes 2}. 
\]
This completes the proof of \eqref{psi-2}. 
\end{proof}

\begin{remark}
Although we do not give the detailed proof, we can show by the induction that it is also true for any $k\in \N$ that 
\[
\n_\th^k \psi_\th(u) \sim D_k(\th) u^k e^{-\g_\th u},\quad u\to \infty,
\]
where $D_k(\th) := C_\th \l[\mu_\th^{-1}\L \dot{g}_\th(-\g_\th)\r]^{\otimes k}$, under [NPC], [MO], [TD($k$)] and [DI($k$)]. 
\end{remark}

Later we need a uniform estimates of $\ddot{\psi}_\th(u)$ as $u\to \infty$. 
\begin{cor}\label{cor:psi-deriv}
\[
\limsup_{u\to \infty}\sup_{\th\in \Theta} \l|u^{-2}e^{\g_\th u}\ddot{\psi}_\th(u)\r| < \infty. 
\]
\end{cor}

\begin{proof}
For this proof, we need to consider a modification of Lemma \ref{lem:key} so that the solution $Z$ to the renewal equation depends on $\th\in \Theta$: 
\[
Z_\th(u) = Z_\th*F_\th(u) + G_\th(u),\quad u\ge 0, 
\]
and assume that the distribution $F_\th$ and the function $G_\th$ satisfy the following conditions: 
\[
\inf_{\th\in \Theta}\l|\int_0^\infty x\,F_\th(\df x)\r| < \infty;\quad \sup_{\th\in \Theta}|G_\th(u)|\lesssim 1, 
\]
and that, for some $k\in \N$, 
\[
\limsup_{u\to \infty}\sup_{\th\in \Theta}\l|\frac{G_\th(u)}{u^{k-1}}\r| < \infty. 
\]
Then it is easy to see by the same argument as the proof of Theorem 4 in Oshime and Shimizu \cite{os17} that 
\[
\limsup_{u\to \infty}\sup_{\th\in \Theta}\l|\frac{Z_\th(u)}{u^k}\r| < \infty. 
\]
Applying this with $k=1$ to the renewal equation \eqref{RE-dot-psi} given in the proof of Theorem \ref{thm:psi-deriv}: 
\[
\wt{\dot{\psi}}_\th(u) = \wt{\dot{\psi}}_\th \star \wt{g}_\th(u) + \l[\wt{\dot{h}}_\th(u) + \wt{\psi}_\th\star \wt{\dot{g}}_\th(u)\r], 
\]
we see that 
\begin{align*}
\inf_{\th\in \Theta}\l|\int_0^\infty x\wt{g}_\th(x)\,\df x\r| < \infty, \quad 
\sup_{\th\in \Theta}\l|\wt{\dot{h}}_\th(u) + \wt{\psi}_\th\star \wt{\dot{g}}_\th(u)\r| \lesssim 1, 
\end{align*}
according to the fact \eqref{property:RE-dot-psi} and the compactness of $\Theta$. Hence we have that 
\[
\limsup_{u\to \infty}\sup_{\th\in \Theta}\l|u^{-1}e^{\g_\th u}\dot{\psi}_\th(u) \r| < \infty. 
\]
\end{proof}

\section{Surplus model and its statistical inference}

\subsection{L\'evy surplus and observations}\label{sec:levy surplus}

In this paper, we consider the L\'evy surplus $R$ given by \eqref{model}: 
\begin{align*}
R_t = u + ct + \s_0 W_t - S_t,\quad t\ge 0, 
\end{align*}
where $S$ is a L\'evy process with the L\'evy measure $\nu_{\a_0}$. Note that $\th_0=(\a_0,D_0=\s_0^2/2)$ is the true value of the parameter. 

Since the process generally has infinitely many jumps in any finite time interval, it seems unrealistic to use a L\'evy process to the surplus model. 
However several view points are possible to consider such a L\'evy model in  insurance context: 
(I) Infinitely many ``small" jumps are interpreted as an approximation to frequently occurred small claims and other extra fees in insurance business, among others; 
(II) The jump process $S$ does not necessarily imply the aggregate claims but a ``macro approximation" of the whole path of surplus; see Shimizu \cite{s09}. 
The former (I) may be commonly recognized aspect in this context. The latter aspect (II) is similar to the one in financial context, where the stock price 
is modeled by the jump process with discrete observation although we do not know whether the jump really exists or does not. 
Anyway, the L\'evy surplus model can be better approximation than the classical one for the real surplus, and such a L\'evy model is becoming popular in recent academism of ruin theory. 

Considering the statistical inference for such a L\'evy surplus, there are several possibilities for available data set.  
\begin{itemize}
\item The value of the surplus at discrete time points on $[0,T]$: 
\[
R^n:=\{R_{ih}\,|\,i=0,1,2,\dots,n,\ h>0\},\quad T:=nh
\]
\item Claims larger than a certain level $\e_T\ge 0$, say $J_S(\e_T,\infty)$, where 
\[
J_S(a,b)=\{\D S_t\,|\,a < \D S_t\le b,\ t\in [0,T]\}. 
\]
\end{itemize}
If we adopt the aspect (I) then we may assume that $R^n\cup J_S(\e_T,\infty)$ is available, and consider the asymptotics such that 
$T\to \infty$ as well as $\e_T\to \e\ge 0$. In particular, if the subordinator is a compound Poisson process, then it would be reasonable to assume $\e_T\equiv 0$. 
 
If we do (II) then it would be reasonable to assume that only $R^n$ is available. 
In this case, estimation of $\th_0$ is possible under different sampling (asymptotic) schemes: 
\begin{itemize}
\item {\it Low-frequency}: $T\to \infty$ as $n\to \infty$; $h\equiv$ fixed. 
\item {\it High-frequency}: $T\equiv$ fixed.; $h\to 0$ as $n\to \infty$.
\item {\it High-frequency with long term}: $T\to \infty$ and $h\to 0$ as $n\to \infty$.
\end{itemize}
See Akritas and Johnson \cite{aj81} and Basawa and Brockwell \cite{bb78} for case (I), and see also Shimziu \cite{s11} in the context of ruin theory. 
Moreover, see Woerner \cite{w01} for case (II), which discusses the inference for discretely observed L\'evy processes under high-frequency setting. 

Later, we consider the case where $T\to \infty$, which is usually needed to estimate $\nu_\a$.

\subsection{Estimating the parameter $\th_0$}
In this paper, we do not discuss how to estimate the parameter $\th_0$ in general because it strongly depends on the model of $\nu_\a$. 
We just suppose that some consistent estimator $\wh{\th}_T\ \toP \th_0$ is given. 
Although the rate of convergence of the estimator can be different in general due to the sampling scheme given in the previous section and the properties of $\nu_{\a_0}$, 
since we are now dealing with the case where 
\[
\int_0^1 z \nu_{\a_0}(\df z) < \infty, 
\]
under which the variation due to jumps are finite, the optimal rate of convergence for estimation of $\a_0$ is 
$\sqrt{T}$ as $T\to \infty$ according to Woerner \cite{w01}, Section 4.2. 

On the other hand, the diffusion parameter $D_0$ is also estimable under both schemes (I) and (II). 
For case (I), we can find an estimator such that $\wh{D}_T =  D_0 + o_p(1/\sqrt{T})$ as $T\to \infty$ under a suitable high frequency conditions; 
see e.g., Jacod \cite{j07}, or Shimizu \cite{s11} in the context of ruin theory. 
For case (II), see e.g., Masuda \cite{ma13} for a quasi-likelihood approach, or Shimizu and Yoshida \cite{sy06} or Shimizu \cite{s06} for ``threshold estimation". 
For example, using the thresholding technique given in \cite{sy06} and \cite{s06}, we can find an estimator such that 
$\sqrt{n}(\wh{D}_n -  D_0)$ is asymptotically normal under $T\to \infty$ and $h\to 0$ as $n\to \infty$. Hence this estimator also satisfies that 
$\wh{D}_n =  D_0 + o_p(1/\sqrt{T})$. 
In any case, we should note that the high frequency assumption is needed to estimate the diffusion parameter $D_0$ separately from the parameters in the jump component. 

Henceforth, we suppose that estimators such that
\begin{align}
\sqrt{T}\begin{pmatrix} 
\wh{\a}_T - \a_0 \\
\wh{D}_T -  D_0
\end{pmatrix}
\toD \begin{pmatrix}{\cal N}_{p}(0,\Sigma_0) \\ 0\end{pmatrix}, \label{th-hat}
\end{align}
as $T\to \infty$ and $h\to 0$, is available, and that the asymptotic variance $\Sigma_0$ is also estimable by $\wh{\Sigma}_T$: 
\begin{align}
\wh{\Sigma}_T\toP \Sigma_0,\quad T\to \infty. 
\end{align}
Some concrete examples are discussed later. 

\subsection{Estimating the adjustment coefficient}

Once the parameters are estimated, we can also estimate the adjustment coefficient $\g_{\th_0}$. 
Neverthless, we can not substitute $\wh{\th}_T$ for $\th_0$ directly such as $\g_{\wh{\th}_T}$ because 
the map $\th\mapsto \g_\th$ is not an explicit function of $\th$, but implicit positive solution to the Lundberg euqation \eqref{l-eq}.

\begin{thm}\label{thm:adj}
Suppose the conditions [NPC], [MO], [TD(1)], and that there exists a statistic $\wh{\g}_T$ uniquely which satisfies the equation
\begin{align}
\k_{\wh{\th}_T}(\wh{\g}_T) = 0,  \label{l-eq-approx}
\end{align}
where $\k_\th(r)$ is the function given in \eqref{l-eq}. 
Then $\wh{\g}_T = \g_{\wh{\th}_T}$, and it holds that 
\[
\sqrt{T}(\wh{\g}_T - \g_0) \toD {\cal N}_1\l(0, \frac{\n_\a \k_{\th_0}^\top(\g_0) \Sigma_0\n_\a \k_{\th_0}(\g_0)}{\l[c + 2D \g_0 - \int_0^\infty z e^{\g_0 z}\,\nu_{\a_0}(\df z)\r]^2}\r) ,\quad n\to \infty. 
\]
\end{thm}

\begin{proof}
In the beginning, we note that $\g_0$ exists uniquely under the conditions [NPC] and [MO], and that $\g_0$ and $\wh{\g}_n$ satisfies the following equalities: 
\begin{align}
&-c \g_0 + D_0 \g_0^2 + \int_0^\infty (e^{\g_0 z} - 1)\,\nu_{\a_0}(\df z) = 0;  \label{adj1} \\
&-c \wh{\g}_T + \wh{D}_T \wh{\g}_T^2 + \int_0^\infty (e^{\wh{\g}_T z} - 1)\,\nu_{\wh{\a}_T}(\df z) = 0.  \label{adj2}
\end{align}
The weak consistency $\wh{\g}_n\toP \g_0$ is obtained by the standard theory of $Z$-estimator, e.g., van der Vaart (1998), Lemma 5.1. 
Hence we only show the asymptotic normality. For that purpose, we may assume in the sequel that 
\begin{align}
\wh{\g}_T\to \g_0,\quad \wh{\a}_T\to \a_0\quad a.s., \label{suppose-a.s.}
\end{align}
by the standard {\it ``sub-sub sequence" argument}, which makes the discussion simple. 

Here we note that, under [MO] and [TD(0)], it follows by Fubini's theorem that 
\[
\int_0^\infty (e^{\g z}-1)\,\nu_{\a}(\df z) = \g\int_0^\infty e^{\g x}\Pi_{\a}(x)\,\df x
\]
for any $(\g,\a)$ that is around $(\g_0,\a_0)$. 
Therefore the above equality holds if $\g=\wh{\g}_T$ or $\a=\wh{\a}_T$ for $T$ large enough. 
We shall take the difference of \eqref{adj1} and \eqref{adj2} and the mean value theorem to obtain that 
\begin{align}
&\quad -c(\wh{\g}_T - \g_0) + \l[D_0\g_0^2 - \wh{D}_T\wh{\g}_T^2\r] \notag \\
&\qquad \quad+ \int_0^\infty (e^{\g_0z}-1)\,\nu_{\a_0}(\df z) - \int_0^\infty (e^{\wh{\g}_Tz}-1)\,\nu_{\wh{\a}_T}(\df z) = 0 \notag \\
\Leftrightarrow &\quad 
-c(\wh{\g}_T - \g_0) + \wh{D}_T\l[\wh{\g}_T^2 - \g_0^2\r] - \g_0^2\l[\wh{D}_T - D_0\r] \notag\\ 
&\qquad \quad - \int_0^\infty (e^{\wh{\g}_Tz}-1)\,\nu_{\wh{\a}_T}(\df z)  + \int_0^\infty (e^{\wh{\g}_Tz}-1)\,\nu_{\a_0}(\df z)\notag \\
&\qquad \quad - \int_0^\infty (e^{\wh{\g}_Tz}-1)\,\nu_{\a_0}(\df z)+ \int_0^\infty (e^{\g_0z}-1)\,\nu_{\a_0}(\df z) = 0 \notag \\
\Leftrightarrow &\quad 
-c(\wh{\g}_T - \g_0) + \wh{D}_T\l[\wh{\g}_T^2 - \g_0^2\r] - \g_0^2\l[\wh{D}_T - D_0\r] \notag\\
&\qquad \quad  + \wh{\g}_T\int_0^\infty e^{\wh{\g}_T x}\l[\Pi_{\wh{\a}_T}(x) - \Pi_{\a_0}(x)\r]\,\df x - \int_0^\infty (e^{\wh{\g}_Tz} - e^{\g_0 z})\,\nu_{\a_0}(\df z) = 0 \notag \\
\Leftrightarrow &\quad 
\sqrt{T}\l(\wh{\g}_T - \g_0\r)\l[c - 2D_0\g_0 - \int_0^\infty \g_0e^{\g_0z}\,\nu_{\a_0}(\df z) + o_p(1)\r] \notag \\
&\qquad \quad = \g_0\int_0^\infty e^{\g_0 x}\n_\a \Pi_{\a^*}(x)\,\df x\cdot \sqrt{T}\l(\wh{\a}_T - \a_0\r), 
\label{eq2}
\end{align}
where $\a^*$ is a random number between $\wh{\a}_T$ and $\a_0$. Note that $\a^*\to \a_0\ a.s.$ by the assumption \eqref{suppose-a.s.}, and that 
$\sqrt{T}\l(\wh{\a}_T - \a_0\r) \toD {\cal N}_p(0,\Sigma_0)$.
Moreover, note that 
\[
\int_0^\infty e^{\g_0 x}\n_\a \Pi_{\a^*}(x)\,\df x = \n_\a \int_0^\infty (e^{\g_0z} - 1)\,\nu_{\a_0}(\df z) = \n_\a\k_\th(\g_0)
\]
under the condition [TD(1)]. 
As a result, Slutsky's lemma ends the proof.  
\end{proof}

\subsection{Main result}

As described in Section \ref{sec:par-inf}, we can apply the delta method to the ruin probability $\psi_{\wh{\th}_T}(u)$ to obtain 
the asymptotic distribution

\begin{thm}\label{thm:main}
Suppose [NPC], [MO], [TD($2$)] and [DI($2$)]. Let $\{u_T\}_{T\ge 0}$ be a real sequence such that, as $T\to \infty$, 
\[
u_T\to \infty,\quad u_T/\sqrt{T}\to 0. 
\]
Moreover let 
\[
\s^*(\th,u):=\l[ \z(\th)^\top \Sigma_0^* \z(\th)\r]^{1/2}\cdot u\,e^{-\g_\th u} \in \R, 
\]
where  $\Sigma_0^* := \begin{pmatrix} \Sigma_0 & \mb{O}_{p,1} \\ \mb{O}_{1,p} & 0\end{pmatrix}$, which is $(p+1)\times (p+1)$-matirix, 
and  
\[
\z(\th) := C_\th \l[\mu_\th^{-1}\L \dot{g}_\th(-\g_\th)\r] \in \R^{p+1}
\]
Then it holds that 
\[
\frac{\sqrt{T}(\psi_{\wh{\th}_T}(u_T) - \psi_\th(u_T))}{\s^*(\wh{\th}_T,u_T)} \toD N(0,1), \quad T\to \infty. 
\]
\end{thm}

\begin{proof}

Note the following decomposition: 
\begin{align*}
\frac{\sqrt{T}(\psi_{\wh{\th}_T}(u_T) - \psi_\th(u_T))}{\s^*(\wh{\th}_T,u_T)} 
&= \frac{\s^*(\th_0,u_T)}{\s^*(\wh{\th}_T,u_T)}\cdot \frac{\sqrt{T}\l(\psi_{\wh{\th}_T}(u_T) - \psi_{\th_0}(u_T)\r)}{\l[\dot{\psi}^\top_{\th_0}(u_T)\Sigma_0^* \dot{\psi}_{\th_0}(u_T)\r]^{1/2}}
\cdot \frac{\l[\dot{\psi}^\top_{\th_0}(u_T)\Sigma_0^* \dot{\psi}_{\th_0}(u_T)\r]^{1/2}}{\s^*(\th_0,u_T)} \\
&=: X_T\cdot Y_T\cdot Z_T, 
\end{align*}
and that we can check by the Lebesgue convergence theorem that $\s^*:\Theta\times \R_+\to \R$ is continuous under the condition [TD(1)]. 
Therefore it follows by the continuous mapping theorem that $X_n\to 1\ \ a.s.$. Moreover, it is clear that $Z_n\to 1$ by Theorem \ref{thm:psi-deriv}. 
Hence the proof ends if we show that $Y_n\toD N(0,1)$. 

As for $Y_n$, it follows by Taylor's formula that 
\begin{align}
Y_n &= \frac{\dot{\psi}_{\th_0}^\top(u_T)\sqrt{T}(\wh{\th}_T - \th_0)}{\l[\dot{\psi}^\top_{\th_0}(u_T)\Sigma_0^* \dot{\psi}_{\th_0}(u_T)\r]^{1/2}}  
+  \sqrt{T}\frac{(\wh{\th}_T - \th_0)^\top \ddot{\psi}_{\th^*_T}(u_T) (\wh{\th}_T - \th_0)}{\l[\dot{\psi}^\top_{\th_0}(u_T)\Sigma_0^* \dot{\psi}_{\th_0}(u_T)\r]^{1/2}}, \label{Yn}
\end{align}
where $\th^*_T$ is a random point between $\wh{\th}_T$ and $\th_0$. It is clear that the first term in the right-hand side of \eqref{Yn} converges in law to the standard normal distribution since 
we are assuming $\sqrt{T}(\wh{\th}_T - \th_0) \to {\cal N}_{p+1}(0,\Sigma_0^*)$. As for the second term, 
we apply Theorem \ref{thm:psi-deriv}, Corollary \ref{cor:psi-deriv} and Theorem \ref{thm:adj} to obtain that 
\begin{align*}
&\sqrt{T}\frac{(\wh{\th}_T - \th_0)^\top \ddot{\psi}_{\th^*_T}(u_T) (\wh{\th}_T - \th_0)}{\l[\dot{\psi}^\top_{\th_0}(u_T)\Sigma_0^* \dot{\psi}_{\th_0}(u_T)\r]^{1/2}} \\
&= \frac{u_T}{\sqrt{T}}\frac{\sqrt{T}(\wh{\th}_T - \th_0)^\top \l(u_T^{-2}e^{\g_{\th^*_T}u_T}\ddot{\psi}_{\th^*_T}(u_T)\r) \sqrt{T}(\wh{\th}_T - \th_0)}
{\l[\l(u_T^{-1}e^{\g_{\th_0}u_T}\dot{\psi}_{\th_0}(u_T)\r)^\top\Sigma_0^* \l(u_T^{-1}e^{\g_{\th_0}u_T}\dot{\psi}_{\th_0}(u_T)\r)\r]^{1/2}} \cdot e^{-(\g_{\th^*_T} - \g_{\th_0})u_T} \\
&= O_p\l(\frac{u_T}{\sqrt{T}}\r) \l[\exp\l\{-\sqrt{T}(\g_{\wh{\th}_T} - \g_{\th_0})\cdot \frac{u_T}{\sqrt{T}}\r\} + o_p(1)\r] \toP 0. 
\end{align*}
Therefore $Y_n\toD N(0,1)$, and the proof is completed. 

\end{proof}

From these results, we immediately obtain the following confidence intervals. 

\begin{cor}\label{cor:CI}
Let $z_\a$ be the upper $\a$-percentile for ${\cal N}(0,1)$, and suppose the same assumptions as in Theorem \ref{thm:main}. 
Then the sequence of intervals 
\[
I_T^\a :=\l[\psi_{\wh{\th}_T}(u_T) - z_{\a/2}\frac{\s^*(\wh{\th}_T,u_T)}{\sqrt{T}},\ \psi_{\wh{\th}_T}(u_T) + z_{\a/2}\frac{\s^*(\wh{\th}_T,u_T)}{\sqrt{T}}\r]
\]
is an asymptotic confidence interval with the confidence level $(1-\a)$ in the sense that 
\[
\lim_{T\to \infty}\P\l(\psi_{\th_0}(u_T) \in I_T^\a\r) = 1-\a. 
\]
\end{cor}

\begin{remark}\label{rem:psi-dot-approx}
In the above $I_n^\a$, we have to still compute $\psi_{\wh{\th}_n}$, which is also difficult to compute. 
In practice, we should approximate it by Theorem \ref{thm:cl-approx}, that is, 
\[
J_T^\a:=\l[C_{\wh{\th}_T}e^{-\g_{\wh{\th}_T}u_T} - z_{\a/2}\frac{\s^*(\wh{\th}_T,u_T)}{\sqrt{T}},\ C_{\wh{\th}_T}e^{-\g_{\wh{\th}_T}u_T} + z_{\a/2}\frac{\s^*(\wh{\th}_T,u_T)}{\sqrt{T}}\r]
\]
This seems a practical confidence interval. Otherwise, we can use some nonparametric estimator of $\psi_\th$; see e.g., Zhang and Yang \cite{zy13,zy14}, Shimizu and Zhang \cite{sz17}. 
\end{remark}

\section{Examples}

\subsection{Exponential claims in the classical model}

Consider the Cram\'er-Lundberg model with exponential claims: 
\[
R_t = u + ct - \sum_{i=1}^{N_t} U_i, 
\]
where $N$ is a Poisson process with the intensity $\la_0>0$, 
$U_i$'s are exponential random variables with mean $\mu_0$, and we put $\th_0=(\mu_0,\la_0)$. 

As is well-known, the ruin probability $\psi_\th(u)$ in this case is written explicitly as 
\[
\psi_\th(u) = \frac{\la\mu}{c}e^{-\g_\th u}, \quad u\ge 0, 
\]
where $\g_\th = 1/\mu - \la/c$ is the adjustment coefficient, and this expression is consistent to the Cram\'er approximation, the right-hand side of \eqref{cl-approx}, 
that is, $C_\th = \la \mu /c$. 
Since 
\begin{align*}
g_\th(x) = \frac{\la}{c}e^{-\frac{x}{\mu}};\quad \wt{g}_\th(x) = \frac{\la}{c}e^{-\frac{\la}{c}x},  
\end{align*}
we have the asymptotic formula for $\dot{\psi}_\th$ by Theorem \ref{thm:psi-deriv} as follows.  
\begin{align}
\dot{\psi}_\th(u) \sim \l(\frac{\la}{c\mu}, \frac{\la\mu}{c^2}\r)^\top u e^{-\g_\th u},\quad u\to \infty. \label{dot-exp}
\end{align}
On the other hand, we can check the above expression by the explicit computation as follows. 
\begin{align*}
\dot{\psi}_\th(u) 
&= \frac{\la\mu}{c}\l(\frac{1}{\mu} + \frac{u}{\mu^2}, \frac{1}{\la} + \frac{u}{c}\r)^\top e^{-\g_\th u} \notag \\
&=  \l(\frac{\la}{c\mu}, \frac{\la\mu}{c^2}\r)^\top u e^{-\g_\th u} + O\l(e^{-\g_\th u}\r),\quad u\to \infty,  
\end{align*}
the first term in the last display is the same as the one given in \eqref{dot-exp}. 

Estimation of parameters is easy. Using the claims data $(U_1,U_2,\dots,U_{N_T})$ on $[0,T]$-time interval, 
we can obtain the maximum likelihood estimator of $\th_0$ as follows: 
\[
\wh{\mu}_T = \frac{1}{N_T}\sum_{i=1}^{N_T} U_i,\quad \wh{\la}_T=\frac{N_T}{T}, 
\]
and can easily check that the estimator $\wh{\th}_T=(\wh{\mu}_T, \wh{\la}_T)$ is asymptotically normal: 
\[
\sqrt{T}\begin{pmatrix} \wh{\mu}_T-\mu_0 \\ \wh{\la}_T - \la_0 \end{pmatrix} \toD {\cal N}_2\l(\mb{O}_{2,1},\rm{diag}(\mu_0^2,\la_0)\r),\quad T\to \infty. 
\]
Therefore we can estimate the adjustment coefficient as 
\[
\wh{\g}_T = \g_{\wh{\th}_T} = \frac{1}{\wh{\mu}_T} - \frac{\wh{\la}}{c}
\]
Noticing that
\[
\s^*(\th,u) = \frac{\la}{c}\sqrt{1 + \frac{\la\mu^2}{c^2}}, 
\]
we have the asymptotic $\a$-confidence interval as follows: 
\[
I^\a_T = \l[\frac{\wh{\la}_T\wh{\mu}_T}{c}e^{-\wh{\g}_T u} \pm z_{\a/2}\frac{u_T\,e^{-\wh{\g}_T u_T}}{\sqrt{T}}\frac{\wh{\la}_T}{c}\sqrt{1 + \frac{\wh{\la}_T\wh{\mu}_T^2}{c^2}}\r]
\]

\subsection{The classical model with diffusion perturbation}

Consider the classical risk model perturbed by diffusions, where the L\'evy subordinator $S$ in \eqref{model} is a compound Poisson process with positive jumps: 
\[
R_t = u + ct + \s_0 W_t - \sum_{i=1}^{N_t} U_i, 
\]
where $W$ is a Wiener process; $\s_0\ge 0$ is a parameter; $U_i$'s are $(0,\infty)$-valued random variables with distribution $F_{\b_0}$ and $\b_0\in \R^q$ is a parameter. 
Put $\a_0 = (\b_0,\la_0)$ and $\th_0=(\a_0,D_0)$. As usual, we assume that $W$, $N$ and $U_i$'s are independent each other. 

In this setting, it would be reasonable to assume that all the jumps $\{U_1,U_2,\dots,U_{N_T}\}$ are observed since each $U_i$ represents $i$th claim amount, and 
the discrete observation $R^n:=\{R_{ih}\,|\,i=0,1,2,\dots,n,\ h>0\}$ with $T=nh$ is available. 
That is, we adopt the sampling scheme (I) described in Section \ref{sec:levy surplus}. 

We can estimate $\la_0$ as the maximum likelihood estimator (MLE) $\wh{\la}_T = N_T/T$. 
Moreover, under certain regularities, we can also estiamte $\b$ as the MLE, say $\wh{\b}_T$. 
Then, as is well known, the asymptotic distribution of $\wh{\a}_T = (\wh{\b}_T,\wh{\la}_T)$ is obtained as follows. 
\[
\sqrt{T}(\wh{\a}_T - \a_0) \toD {\cal N}_{q+1}\l(\mb{O}_{q+1,1},\Sigma_0\r),\quad \Sigma_0=\begin{pmatrix} I^{-1}(\b_0) & \mb{O}_{q,1}\\ \mb{O}_{1,q} & \la_0\end{pmatrix}, 
\]
where $I(\b_0)$ is the Fisher information matrix for $\b_0$ while its invertibility is assumed. 

According to Lemma 3.1 with Remark 3.2 in Shimizu \cite{s11}, we can etimate the diffusion parameter $D_0$ as follows: 
\[
\wh{D}_T = \frac{1}{2T}\l[\sum_{i=1}^n |\D_i^n R|^2 - \sum_{i=1}^{N_s} U_i^2\r], 
\]
where $\D_i^n R:=R_{ih} - R_{(i-1)h}$. Then it holds for any $s \in [0,T]$ that 
\[
\sqrt{T}(\wh{D}_T - D_0) \toP 0
\]
as $T=nh\to \infty$, $h\to 0$ with $nh^2\to 0$. Hence we have that 
\[
\sqrt{T}\begin{pmatrix} \wh{\a}_T - \a_0 \\ \wh{D}_T - D_0 \end{pmatrix} \toD {\cal N}_{q+2}(0,\Sigma_0^*),\quad 
\Sigma_0^*=\begin{pmatrix} I^{-1}(\b_0) & \mb{O}_{q,1} & 0 \\ \mb{O}_{1,q} & \la_0 & 0 \\ 0& 0& 0\end{pmatrix}, 
\]

The estimator of the adjustment coefficient $\g_0$ is given as a positive solution to 
\[
-cr + \wh{D}_T^2r^2 + \wh{\la}_T\int_0^\infty (e^{rz} - 1)F_{\wh{\b}_T}(\df z) = 0. 
\]
As a special case, we shall consider the exponential claims: 
\[
F_{\b_0}(x) = 1 - e^{x/\mu_0}\quad (\b_0=\mu_0)
\]
Then the above equation is given by 
\[
-cr + + \wh{D}_T^2r^2 + \frac{\wh{\la}_T\wh{\mu}}{1-\wh{\mu}_T r} = 0
\]

In this model, we can compute $g_\th$ and $\dot{g}_\th$ in Theorem \ref{thm:psi-deriv} as follows: 
\begin{align*}
g_\th(x) = \frac{\la\mu}{D - c\mu}\l(e^{-\frac{c}{D}x} - e^{-\frac{x}{\mu}}\r)
\end{align*}
and 
\begin{align*}
\frac{\p}{\p \mu}g_\th(x) &= \frac{\la}{\mu(c\mu-D)}xe^{-\frac{x}{\mu}} + \frac{\la D}{(c\mu-D)^2}\l(e^{-\frac{c}{D}x} - e^{-\frac{x}{\mu}}\r); \\
\frac{\p}{\p \la}g_\th(x) &= \frac{\mu}{c\mu-D}\l(e^{-\frac{x}{\mu}} - e^{-\frac{c}{D}x}\r); \\
\frac{\p}{\p D}g_\th(x) &= \frac{\la\mu}{(c\mu-D)^2}\l(e^{-\frac{x}{\mu}} - e^{-\frac{c}{D}x}\r) - \frac{\la\mu c}{D^2(c\mu - D)}xe^{-\frac{c}{D}x}. 
\end{align*}
Therefore we have that 
\begin{align*}
\z(\th) = C_\th\l[\mu_\th^{-1}\L \dot{g}_\th(-\g_\th)\r] := C_\th \mu_\th^{-1} (L_\mu,L_\la,L_D)^\top, 
\end{align*}
where, by $p_{\mu,\th}=(1/\mu-\g_\th)^{-1}$, $p_{D,\th}=(c/D - \g_\th)^{-1}$, 
\begin{align*}
L_\mu&=\frac{\la D}{(c\mu - D)^2}(p_{\mu,\th} - p_{D,\th}) + \frac{\la}{\mu(c\mu -D)}p_{\mu,\th}^2, \\
L_\la&=\frac{\mu}{c\mu - D}(p_{\mu,\th} - p_{D,\th}), \\
L_{D}&=\frac{\la\mu}{(c\mu - D)^2}(p_{\mu,\th} - p_{D,\th}) - \frac{\la\mu c}{D^2(c\mu - D)^2}p_{D,\th}^2, \\
\mu_\th &= \frac{\la\mu}{c\mu - D}(p_{\mu,\th}^2 - p_{D,\th}^2), \\
C_\th&= \frac{\la\mu}{\la\mu\l(1-\mu\g\r)^{-2} - c + 2D\g_\th}. 
\end{align*}
Using the estimator $\wh{\th}_T$, we have that 
\begin{align*}
\s^*(\wh{\th}_T,u) &= \z(\wh{\th}_T)^\top \begin{pmatrix}\wh{\mu}_T & 0 & 0\\ 0 & \wh{\la}_T & 0 \\ 0&0&0\end{pmatrix} \z(\wh{\th}_T) 
= \wh{\mu}_T\wh{L}_\mu^2 + \wh{\la}_T\wh{L}_\la^2, 
\end{align*}
where $\wh{L}_\mu$ and $\wh{L}_\la$ is estimators of $L_\mu$ and $L_\la$, respectively, such that the unknown parameters therein are replaced with their estimators. 
Therefore the asymptotic $(1-\a)$-confidence interval $J_T^\a$ is given as follows: 
\[
J_T^\a:=\l[C_{\wh{\th}_T}e^{-\g_{\wh{\th}_T}u_T} \pm z_{\a/2}\frac{\wh{\mu}_T\wh{L}_\mu^2 + \wh{\la}_T\wh{L}_\la^2}{\sqrt{T}}u_Te^{-\wh{\g}_T u_T}\r]
\]

\subsection{Gamma subordinator}

Let us consider a L\'evy insurance risk model 
\begin{align}
R_t = u + ct - Z_t, \label{subord-model}
\end{align}
where $Z$ is a gamma process: 
\[
\P(Z_t\in \df x) = \frac{b^{a t}}{\G(a t)} x^{a t -1} e^{-b x}, \quad z\ge 0, 
\]
which has a monotonically increasing path with infinitely many jumps in any finite time interval. 
The L\'evy measure of $X$ is given as 
\[
\wt{\nu}_\a(z)\,\df z :=a z^{-1}e^{-b z}\,\df z,\quad z> 0,
\]
where $\a=(a,b)$ is unknown parameters. 
To estimate the parameter $\a$, we suppose the claims whose sizes are larger than $\e_T>0$: $J_Z(\e_T,\infty) = \{\D Z_t\,|\,\D Z_t > \e_T\}$. 
Since $N_T(\e_T):=\# J_Z(\e_T,\infty) < \infty$ for each $\e_T>0$, we put the jump size as 
\[
J_Z(\e_T,\infty)=\{U_k(\e_T)\,|\,k=1,2,\dots,N_T(\e_T)\}, 
\]
where $U_k(\e_T)$ is the $k$th jump size that is larger than $\e_T$. 

According to Basawa and Brockwell \cite{bb78}, the likelihood function is given by 
\[
L(\a) = \exp\l(-t \int_{\e_T}^\infty \wt{\nu}_\a(z)\,\df z \r)\prod_{k=1}^{N_T(\e_T)} \wt{\nu}_\a(U_k(\e_T)), 
\]
and the maximum likelihood estimator $\wh{\a}_T = (\wh{a}_T,\wh{b}_T)$ is given by solving the equation that 
\[
\int_0^\infty \frac{\wh{b}_Te^{-\wh{b}_T u}}{u + \e_T}\,\df u = \frac{N_T(\e_T)}{\sum_{k=1}^{N_T(\e_T)} U_k(\e_T)},\quad \wh{a}_T = \frac{\wh{b}_T}{T}e^{\wh{b}_T \e_T} \sum_{k=1}^{N_T(\e_T)} U_k(\e_T). 
\]
Then, assuming that $\e_T\to \e>0$, it holds that 
\[
\sqrt{T}(\wh{\a}_T - \a_0) \toD {\cal N}\l(\begin{pmatrix}0\\0\end{pmatrix},\begin{pmatrix}\s_{aa} & \s_{ab}\\ \s_{ba} & \s_{bb}\end{pmatrix} \r), 
\]
where 
\begin{align*}
\s_{aa} &= ae^{-b \e}(1 + b\e)(b \xi)^{-1},\quad \s_{bb}= (a^2\xi)^{-1}\int_\e^\infty \wt{\nu}_\a(z)\,\df z,\quad \s_{ab}=-e^{-b\e}(\b\xi)^{-1}, \\
\xi &= b^{-2}e^{-b\e}\l[(1+b\e)a^{-1}\int_\e^\infty \wt{\nu}_\a(z)\,\df z - e^{-b\e}\r]. 
\end{align*}

Assuming that the adjustment coefficient $\g_\th$ exists and satisfies that $\g_\th < b$. 





\end{document}